\documentclass[11pt, a4paper, reqno]{amsart}
\usepackage{ifdraft}
\usepackage[T1]{fontenc}
\usepackage{lmodern}
\usepackage{microtype}
\usepackage{enumerate}
\usepackage[final]{hyperref}
\usepackage[abbrev,backrefs]{amsrefs}
\usepackage{amssymb}
\usepackage{cmdtrack}

\newtheorem{theorem}{Theorem}

\newtheorem{proposition}{Proposition}[section]
\newtheorem{claim}{Claim}

\theoremstyle{definition}

\numberwithin{equation}{section}

\newcommand{\resetclaim}{\setcounter{claim}{0}}
\newenvironment{proofclaim}
  [1]
  [Proof of the claim]
  {\begin{proof}[#1]}
  {\end{proof}}

\newcommand{\N}{{\mathbb N}}
\newcommand{\Rset}{{\mathbb R}}

\newcommand{\st}{\;:\;}

\newcommand{\abs}[1]{\lvert #1 \rvert}
\newcommand{\bigabs}[1]{\bigl\lvert #1 \bigr\rvert}

\newcommand{\dualprod}[2]{\langle #1, #2 \rangle}
\newcommand{\norm}[1]{\lVert #1 \rVert}

\newcommand{\downto}{\searrow}

\usepackage{refcount}

\newcounter{cte}

\newcommand{\Constant}{\refstepcounter{cte} C_{\thecte}}

\newcommand{\SameConstant}{C_{\thecte}}

\allowdisplaybreaks

\title[Nodal solutions for the quadratic Choquard equation]{Least action nodal solutions\\ for the quadratic Choquard equation}

\author{Marco Ghimenti}
\address{Universit\`a di Pisa\\
Dipartimento di Matematica\\
Largo B. Pontecorvo 5\\
56100 Pisa\\
Italy}
\email{marco.ghimenti@dma.unipi.it}

\author{Vitaly Moroz}
\address{Swansea University\\ Department of Mathematics\\ Singleton Park\\
Swansea\\ SA2~8PP\\ Wales, United Kingdom}
\email{V.Moroz@swansea.ac.uk}

\author{Jean Van Schaftingen}
\address{Universit\'e Catholique de Louvain\\
Institut de Recherche en Math\'ematique et Phy\-sique\\
Chemin du Cyclotron 2
bte L7.01.01\\ 1348 Louvain-la-Neuve \\
Belgium}
\email{Jean.VanSchaftingen@UCLouvain.be}

\keywords{Stationary nonlinear Schr\"odinger--Newton equation;
stationary Hartree equation; nodal Nehari set;
concentration-compactness.}

\subjclass[2010]{35J91 (35J20, 35Q55)}

\date{\today}

\begin{document}

\begin{abstract}
We prove the existence of a minimal action nodal solution for the quadratic Choquard equation
\begin{equation*}
 -\Delta u + u  = \bigl(I_\alpha \ast \abs{u}^2\bigr)u \quad\text{in \(\Rset^N\)},
\end{equation*}
where $I_\alpha$ is the Riesz potential of order $\alpha\in(0,N)$.
The solution is constructed as the limit of minimal action nodal solutions for the
nonlinear Choquard equations
\begin{equation*}
 -\Delta u + u  = \bigl(I_\alpha \ast \abs{u}^p\bigr)|u|^{p-2}u \quad\text{in \(\Rset^N\)}
\end{equation*}
when $p\downto 2$. The existence of minimal action nodal solutions for $p>2$ can be proved
using a variational minimax procedure over Nehari nodal set.
No minimal action nodal solutions exist when $p<2$.
\end{abstract}
\maketitle

\section{Introduction}

We study least action nodal solutions of the quadratic Choquard equation
\begin{equation}
\label{eqChoquard}
\tag{$\mathcal{C}_2$}
 -\Delta u + u  = \bigl(I_\alpha \ast \abs{u}^2\bigr) u \quad\text{in \(\Rset^N\)},
\end{equation}
for \(N \in \N\) and \(\alpha \in (0, N)\). Here \(I_\alpha : \Rset^N \to \Rset\) is the Riesz potential
defined for each \(x \in \Rset^N \setminus \{0\}\) by
\begin{align*}
I_\alpha (x) &= \frac{A_\alpha}{\abs{x}^{N - \alpha}}, &
&\text{where }
&A_\alpha = \frac{\Gamma(\tfrac{N-\alpha}{2})}
                   {\Gamma(\tfrac{\alpha}{2})\pi^{N/2}2^{\alpha} }.
\end{align*}
For \(N = 3\) and \(\alpha = 2\) equation \eqref{eqChoquard} is the  \emph{Choquard--Pekar equation}
which goes back to the 1954's work by S.\thinspace I.\thinspace Pekar  on quantum theory of a polaron at rest \citelist{\cite{Pekar1954}\cite{DevreeseAlexandrov2009}*{Section 2.1}}
and to 1976's model of P.\thinspace Choquard of an electron trapped in its own hole, in an approximation to Hartree-Fock theory of one-component plasma \cite{Lieb1977}. In the 1990's the same equation reemerged as a model of self-gravitating matter  \citelist{\cite{KRWJones1995newtonian}\cite{Moroz-Penrose-Tod-1998}}
and is known in that context as the \emph{Schr\"odinger--Newton equation}.
Equation \eqref{eqChoquard} is also studied as a nonrelativistic model of boson stars \citelist{\cite{Frohlich-Lenzmann}\cite{Lenzmann-2009}}.

Mathematically, the existence and some qualitative properties of solutions of Choquard equation \eqref{eqChoquard}
have been studied by variational methods in the early 1980s,
see \citelist{\cite{Lieb1977}\cite{Lions1980}\cite{Menzala1980}\cite{Lions1984-1}*{Chapter III}} for earlier work on the problem.
Recently, nonlinear Choquard equation
\begin{equation}
\label{eqChoquard-p}
\tag{$\mathcal{C}_p$}
 -\Delta u + u  = \bigl(I_\alpha \ast \abs{u}^p\bigr) \abs{u}^{p - 2} u \quad\text{in \(\Rset^N\)},
\end{equation}
with a parameter $p>1$ attracted interest of mathematicians, see
\citelist{\cite{Ma-Zhao-2010}\cite{CingolaniClappSecchi2012}\cite{CingolaniSecchi}\cite{ClappSalazar}\cite{MorozVanSchaftingen13}\cite{Ghimenti-soliton}} and further references therein; see also \citelist{\cite{CingolaniSecchi-2015}\cite{fractionalChoquard}} for modifications of \eqref{eqChoquard-p} involving fractional Laplacian.
Most of the recent works on Choquard type equations \eqref{eqChoquard-p} so far were dedicated to the study of positive solutions.
Nodal solutions were studied in \citelist{\cite{CingolaniClappSecchi2012}\cite{CingolaniSecchi}\cite{ClappSalazar}\cite{GhimentiVanSchaftingen}\cite{Weth2001}*{theorem 9.5}}.

Equation \eqref{eqChoquard-p} is the Euler equation of the Choquard action functional \(\mathcal{A}_p\)
which is defined for each function \(u\) in the Sobolev space
\(H^1 (\Rset^N)\) by
\[
 \mathcal{A}_p (u) = \frac{1}{2} \int_{\Rset^N} \abs{\nabla u}^2 + \abs{u}^2 - \frac{1}{2 p} \int_{\Rset^N} \bigl(I_\alpha \ast \abs{u}^p\bigr) \abs{u}^p.
\]
By the Hardy--Littlewood--Sobolev inequality, if \(s
\in (1, \frac{N}{\alpha})\) then for
every \(v \in L^s (\Rset^N)\), \(I_\alpha \ast v\in L^\frac{N s}{N - \alpha s}
(\Rset^N)\) and
\begin{equation}
\label{eqHLS}
 \int_{\Rset^N} \abs{I_\alpha \ast v}^\frac{N s}{N - \alpha s} \le C
\Bigl(\int_{\Rset^N} \abs{v}^s \Bigr)^\frac{N}{N - \alpha s},
\end{equation} (see for example \cite{LiebLoss2001}*{theorem 4.3}).
In view of the classical Sobolev embedding, the action functional \(\mathcal{A}_p\) is
well-defined and continuously differentiable if and only if
\[
 \frac{N - 2}{N + \alpha} \le \frac{1}{p} \le \frac{N}{N + \alpha}.
\]
By testing equation~\eqref{eqChoquard-p} against \(u\), the natural
\emph{Nehari constraint}
$\dualprod{\mathcal{A}' (u)}{u} = 0$ appears.
Then positive solutions of \eqref{eqChoquard-p} can be obtained by studying the infimum
\[
  c_{0, p} = \inf\, \bigl\{\mathcal{A}_p (u) \st u \in \mathcal{N}_{0, p}\bigr\}
\]
over the \emph{Nehari manifold}
\[
 \mathcal{N}_{0, p} = \bigl\{u \in H^1 (\Rset^N) \setminus \{0\} \st \dualprod{\mathcal{A}_p' (u)}{u} = 0\}.
\]
It turns out that the infimum $c_{0, p}$ is achieved when
\[
 \frac{N - 2}{N + \alpha} < \frac{1}{p} < \frac{N}{N + \alpha},
\]
and these assumptions are optimal \citelist{\cite{Lieb1977}\cite{Lions1980}\cite{MorozVanSchaftingen13}}.

Besides positive solutions minimising $c_{0, p}$, which are known as {\em groundstates} or {\em least action solutions},
additional solutions can be constructed by several variational construction.
In particular, one can consider \emph{least action nodal solutions}, the sign-changing counterpart of least actions solutions.
One way to search for such solutions is to consider the infimum
\[
  c_{\mathrm{nod}, p} = \inf\, \bigl\{\mathcal{A}_p (u) \st u \in \mathcal{N}_{0, p}\bigr\}
\]
over the \emph{Nehari nodal set}
\begin{multline*}
 \mathcal{N}_{\mathrm{nod}, p}
 =\bigl\{ u \in H^1 (\Rset^N) \st u^+ \ne 0 \ne u^-,\,\\
           \dualprod{\mathcal{A}_p'(u)}{u^+} = 0
           \text{ and } \dualprod{\mathcal{A}_p'(u)}{u^-} = 0\bigr\},
\end{multline*}
where \(u = u^+ - u^-\).
Such construction has been performed for local elliptic problems on bounded domains of \(\Rset^N\),
see \citelist{\cite{CeramiSoliminiStruwe1986}\cite{CastroCossioNeuberger1997}\cite{CastorCossioNeuberger1998}},
whereas the approach fails for the nonlinear Schr\"odinger equation
\begin{equation}\label{e-NLS}
-\Delta u+u=\abs{u}^{2p - 2} u \quad\text{in \(\Rset^N\)},
\end{equation}
which has no least action nodal solutions \citelist{\cite{AckermannWeth2005}*{lemma 2.4}\cite{Weth2006}\cite{GhimentiVanSchaftingen}}.
Moreover, the least action energy on the Nehari nodal set is not approximated by nodal solutions of \eqref{e-NLS},
see \cite{Weth2006}*{theorem 1.5}.

Surprisingly, it has been proved that unlike its local counterpart \eqref{e-NLS},
the nonlocal Choquard equation \eqref{eqChoquard-p} admits least action nodal solutions when
\[
 \frac{N - 2}{N + \alpha} < \frac{1}{p} < \frac{1}{2}
\]
\cite{GhimentiVanSchaftingen}*{theorem 2}; while the infimum $c_{\mathrm{nod}, p}$ is not achieved when
\[
 \frac{1}{2} < \frac{1}{p} < \frac{N}{N + \alpha},
\]
because \(c_{\mathrm{nod}, p} = c_{0, p}\) \cite{GhimentiVanSchaftingen}*{theorem 3}.
The borderline quadratic case \(p = 2\) was not covered by either existence or non-existence proofs in \cite{GhimentiVanSchaftingen}, because of the possible degeneracy of the minimax reformulation of the problem that was introduced (see  \cite{GhimentiVanSchaftingen}*{eq.~(3.3)})
and of difficulties in controlling the norms of the positive and negative parts of Palais--Smale sequences.

The goal of the present work is to study the existence of least action nodal solutions for Choquard equation \eqref{eqChoquard-p}
in the physically most relevant quadratic case \(p = 2\). Because the minimax procedure for capturing  $c_{\mathrm{nod}, p}$ introduced in \cite{GhimentiVanSchaftingen} apparently fails for $p=2$, a different approach is needed.
Instead of directly minimizing $c_{\mathrm{nod}, 2}$, our strategy will be to employ Choquard equations \eqref{eqChoquard-p} with $p>2$ as a regularisation family for the quadratic Choquard equation \eqref{eqChoquard} and to pass to the limit when $p\downto 2$.
Our main result is the following.

\begin{theorem}
  \label{theoremMain}
  If \(N \in \mathbb N\) and \(\alpha \in ((N - 4)^+, N)\),
  then there exists a weak solution \(u \in H^1 (\Rset^N)\) of Choquard equation \eqref{eqChoquard}
  such that
  \(u^+ \ne 0 \ne u^-\) and   \(\mathcal{A}_2 (u) = c_{\mathrm{nod},2}\).
\end{theorem}

The constructed nodal solution $u$ is regular, that is,  $u\in L^1(\Rset^N)\cap C^2(\Rset^N)$, see \cite{MorozVanSchaftingen13}*{proposition 4.1}.
The conditions of the theorem are optimal, as for \(\alpha\not\in ((N - 4)^+, N)\) no sufficiently regular solutions to \eqref{eqChoquard} exist in $H^1 (\Rset^N)$ by the Poho\v zaev identity \cite{MorozVanSchaftingen13}*{theorem 2}.

In order to prove theorem \eqref{theoremMain}, we will approximate a least action nodal solution of quadratic
Choquard equation \eqref{eqChoquard} by renormalised least action nodal solution of \eqref{eqChoquard-p} with $p\downto 2$.
To do this, in section~\ref{section2} we first establish continuity of the energy level \(c_{0, p}\) with respect to $p$.
Then in section~\ref{section3} we prove theorem \ref{theoremMain} by showing that as \(p \downto 2\), positive and negative parts
of the renormalised least action nodal solutions of \eqref{eqChoquard-p} do not vanish and do not diverge apart from each other.

\section{Continuity of the critical levels}
\label{section2}

In the course of the proof of theorem~\ref{theoremMain}, we will need the following strict inequality on critical levels.

\begin{proposition}
\label{propositionStrictInequality}
If \(\frac{N - 2}{N + \alpha} < \frac{1}{p} < \frac{N}{N + \alpha}\), then
\[
  c_{\mathrm{nod}, p} < 2 c_{0, p}.
\]
\end{proposition}

This follows directly from \cite{GhimentiVanSchaftingen}*{propositions 2.4 and 3.7}.
The construction in the latter reference is done by taking translated pa ositive and a negative copy of the groundstate
of \eqref{eqChoquard-p} and by estimating carefully the balance between the truncation and the effect of the Riesz potential interaction. When \(p < 2\), it is known that \(c_{\mathrm{nod}, p} = c_{0, p}\) \cite{GhimentiVanSchaftingen}*{theorem 3} and proposition~\ref{propositionStrictInequality} loses its interest.

Because we shall approximate the quadratic case \(p = 2\) by \(p > 2\), it will be useful to have some information about the continuity of \(c_{0, p}\).

\begin{proposition}
\label{continuityGroundstate}
The function
\(
 p \in (\frac{N + \alpha}{N}, \frac{N + \alpha}{(N - 2)_+}) \mapsto c_{0, p} \in \Rset
\)
is continuous.
\end{proposition}
\begin{proof}
It can be observed that
\[
 c_{0, p} = \inf \Biggl\{ \Bigl(\frac{1}{2} - \frac{1}{2 p}\Bigr) \Biggl(\frac{\Bigl(\displaystyle \int_{\Rset^N} \abs{\nabla u}^2 + \abs{u}^2\Bigr)^\frac{p}{p - 1}}{\Bigl(\displaystyle \int_{\Rset^N} \bigl(I_\alpha \ast \abs{u}^p\bigr) \abs{u}^p\Bigr)^\frac{1}{p - 1}} \Biggr) \st u \in H^1 (\Rset^N) \setminus \{0\} \Biggr\}.
\]
Since for every \(u \in H^1 (\Rset^N)\), the function
\[
 p \in (\tfrac{N + \alpha}{N}, \tfrac{N + \alpha}{(N - 2)_+}) \mapsto \int_{\Rset^N} \bigl(I_\alpha \ast \abs{u}^p\bigr) \abs{u}^p
\]
is continuous, the function \(p \mapsto c_{0, p}\) is then upper semicontinuous as an infimum of continuous functions.

We now consider the more delicate question of the lower semicontinuity. There exists a family of functions \(u_p \in H^1 (\Rset^N)\) such that \eqref{eqChoquard-p} holds
and \(\mathcal{A}_p (u_p) = c_{0, p}\).
By the upper semicontinuity, it follows that the function \(p \in (\frac{N + \alpha}{N}, \frac{N + \alpha}{(N - 2)_+}) \mapsto u_p\in H^1 (\Rset^N)\) is locally bounded.

In view of the equation \eqref{eqChoquard-p} and of the Hardy--Littlewood--Sobolev inequality,
we have
\[
 \int_{\Rset^N} \abs{\nabla u_p}^2 + \abs{u_p}^2
 = \int_{\Rset^N} \bigl(I_\alpha \ast \abs{u_p}^p\bigr) \abs{u_p}^p
 \le \Constant \Bigl(\int_{\Rset^N} \abs{u_p}^\frac{2 N p}{N + \alpha}\Bigr)^\frac{N + \alpha}{N},
\]
where constant $C_1$ could be chosen uniformly bounded when \(p\) remains in a compact subset of \((\frac{N + \alpha}{N}, \frac{N + \alpha}{(N - 2)_+})\).
This implies that \citelist{\cite{Lions1984CC2}*{lemma I.1}\cite{Willem1996}*{lemma
1.21}\cite{MorozVanSchaftingen13}*{lemma 2.3}\cite{VanSchaftingen2014}*{(2.4)}}
\begin{multline*}
 \int_{\Rset^N} \abs{\nabla u_p}^2 + \abs{u_p}^2\\
 \le \Constant \Bigl(\int_{\Rset^N} \abs{\nabla u_p}^2 + \abs{u_p}^2 \Bigr)^\frac{N + \alpha}{N} \Bigl( \sup_{a \in \Rset^N} \int_{B_1 (a)} \abs{u_p}^\frac{2 N p}{N + \alpha} \Bigr)^{(N + \alpha)\bigl(\frac{1}{N} - \frac{N + \alpha}{p} \bigr)},
\end{multline*}
where \(\SameConstant\) can be also chosen uniformly bounded when \(p\) is in a compact subset of \((\frac{N + \alpha}{N}, \frac{N + \alpha}{(N - 2)_+})\).
Up to a translation in \(\Rset^N\), we can thus assume that the function
\[
  p \in (\tfrac{N + \alpha}{N}, \tfrac{N + \alpha}{(N - 2)_+}) \mapsto \int_{B_1} \abs{u}^\frac{2 N p}{N + \alpha}
\]
is locally bounded away from \(0\).

We assume now \((p_n)_{n \in \N}\) to be a sequence in the interval \((\tfrac{N + \alpha}{N}, \tfrac{N + \alpha}{(N - 2)_+})\) that converges to \(p_* \in (\tfrac{N + \alpha}{N}, \tfrac{N + \alpha}{(N - 2)_+})\). Since the sequence \((u_{p_n})_{n \in \N}\) is bounded in the space \(H^1 (\Rset^N)\), there exists a sequence \((n_k)_{k \in \N}\) diverging to infinity and \(u_* \in H^1 (\Rset^N)\) such that the subsequence \((u_{p_{n_k}})_{k \in \N}\) converges weakly in \(H^1 (\Rset^N)\) to \(u_*\).
Moreover, we have by the Rellich--Kondrachov compactness theorem
\[
 \int_{B_1} \abs{u_*}^\frac{2 N p_*}{N + \alpha}
 = \lim_{k \to \infty} \int_{B_1} \abs{u_{p_{n_k}}}^\frac{2 N p}{N + \alpha} > 0.
\]
Thus \(u_* \ne 0\) and \(u_*\) satisfies
\[
 -\Delta u_* + u_* = \bigl(I_\alpha \ast \abs{u}^{p_*}\bigr) \abs{u}^{p_* - 2} u.
\]
We have thus
\[
\begin{split}
 c_{0, p_*} \le \mathcal{A}_{p_*} (u_*)
 &= \Bigl(\frac{1}{2} - \frac{1}{2 p_*}\Bigr)
 \int_{\Rset^N} \abs{\nabla u_*}^2 + \abs{u_*}^2\\
 &\le \liminf_{k \to \infty} \Bigl(\frac{1}{2} - \frac{1}{2 p_{n_k}}\Bigr)
 \int_{\Rset^N} \abs{\nabla u_{p_{n_k}}}^2 + \abs{u_{p_{n_k}}}^2\\
 &= \liminf_{k \to \infty} \mathcal{A}_{p_{n_k}} (u_{p_{n_k}})
 = \liminf_{k \to \infty} c_{0, p_{n_k}}.
\end{split}
\]
Since the sequence \((p_n)_{n \in \N}\) is arbitrary, this proves the lower semicontinuity.
\end{proof}

\section{Proof of the main theorem}
\label{section3}

\begin{proof}%
  [Proof of theorem~\ref{theoremMain}]%
  \resetclaim
We shall successively construct a family of solutions of \eqref{eqChoquard}, prove that neither the positive nor the negative part of this family goes to \(0\), and show that the negative part and the positive part cannot diverge from each other as \(p \to 2\). The theorem will then follow from a classical weak convergence and local compactness argument.

\begin{claim}
\label{claimConstructionSequence}
There exists a family \((u_p)_{p \in (2, \frac{N + \alpha}{N - 2})}\) in \(H^1 (\Rset^N)\) such that for each \(p \in (2, \frac{N + \alpha}{N - 2})\), the function \(u_p\) changes sign and
\[
 -\Delta u_p + u_p = \bigl(I_\alpha \ast \abs{u_p}^p\bigr) \abs{u_p}^{p - 2} u_p.
\]
Moreover
\[
  \limsup_{p \to 2} \mathcal{A}_{p} (u_p) \le c_{\mathrm{nod},2}
\]
and
\[
 \limsup_{p \to 2} \int_{\Rset^N} \abs{\nabla u_p}^2 + \abs{u_p}^2 \le
 4\, c_{\mathrm{nod},2}.
\]

\end{claim}

The claim implies the uniform boundedness in \(H^1 (\Rset^N)\) of the solutions \(u_p\) since \(c_{\mathrm{nod}, 2} \le c_{\mathrm{odd}, 2} < 2 c_{0,2} < \infty\) \cite{GhimentiVanSchaftingen}.

\begin{proofclaim}
The existence of a function \(u_p \in H^1 (\Rset^N)\) that changes sign and satisfies Choquard equation \eqref{eqChoquard-p}
has been proved in \cite{GhimentiVanSchaftingen}*{theorem 2}. Moreover,
\[
 \Bigl(\frac{1}{2} - \frac{1}{2p}\Bigr) \int_{\Rset^N} \abs{\nabla u_p}^2 + \abs{u_p}^2 =  \mathcal{A}_p (u_p) = c_{\mathrm{nod}, p}.
\]
It remains to obtain some upper asymptotics on \(c_{\mathrm{nod}, p}\) as \(p \downto 2\).

Let \(w \in \mathcal{N}_{2, \mathrm{nod}}\) and define
\(w_p = t_{+, p}^{{1}/{p}} w^+ - t_{-, p}^{{1}/{p}} w^-\), where \((t_{+, p}, t_{-, p}) \in [0, \infty)^2\) is the unique maximizer of the concave function
\begin{multline*}
 (t_+, t_-) \in [0, \infty)^2
 \mapsto E_p (t_+, t_-) = \mathcal{A}_p (t_+^\frac{1}{p} w^+ - t_-^\frac{1}{p} w^-)\\
 = \frac{t_+^\frac{2}{p}}{2} \int_{\Rset^N} \abs{\nabla w^+}^2+\abs{w^+}^2
 + \frac{t_-^\frac{2}{p}}{2} \int_{\Rset^N} \abs{\nabla w^-}^2+\abs{w^-}^2\\
 - \frac{1}{2 p} \int_{\Rset^N} \bigabs{I_{\alpha/2} \ast \bigl(t_+ \abs{w^+}^p + t_- \abs{w^-}^p)}^2.
\end{multline*}
Since \((t_{+, p}, t_{-, p}) \in (0, \infty)\), we have \(w_p \in \mathcal{N}_{\mathrm{nod}, p}\) and
\[
 c_{\mathrm{nod}, p} \le \mathcal{A}_p (w_p).
\]
Since \(E_p (t_+, t_-) \to -\infty\) as \((t_+, t_-) \to \infty\) uniformly in \(p\) in bounded sets and since \(E_p \to E_2\) as \(p \to 2\) uniformly over compact subsets of \([0, \infty)^2\), we have \(t_{\pm, p} \to 1\) as \(p \to 2\). Therefore
\[
 \lim_{p \to 2} \mathcal{A}_p (w_p) = \mathcal{A}_2 (w).
\]
Since the function \(w \in \mathcal{N}_{\mathrm{nod}, 2}\) is arbitrary, we deduce that
\[
 \limsup_{p \to 2} c_{\mathrm{nod}, p} \le  c_{\mathrm{nod}, 2},
\]
from which the upper bounds of the claim follow.
\end{proofclaim}

\begin{claim}
\label{claimNonzero}
\[
 \liminf_{p \to 2} \int_{\Rset^N} \abs{\nabla u_p^\pm}^2 + \abs{u_p^\pm}^2
 = \liminf_{p \to 2} \int_{\Rset^N} \bigl(I_\alpha \ast \abs{u_p}^p\bigr) \abs{u_p^\pm}^p > 0.
\]
\end{claim}
\begin{proofclaim}
We first compute by the Hardy--Littlewood--Sobolev inequality
\[
\begin{split}
 \int_{\Rset^N} \abs{\nabla u_p}^2 + \abs{u_p}^2
 = \int_{\Rset^N} \bigl(I_\alpha \ast \abs{u_p}^p\bigr) \abs{u_p}^p
 &\le \Constant \Bigl(\int_{\Rset^N} \abs{u_p}^{\frac{2 N p}{N + \alpha}} \Bigr)^{1 + \frac{\alpha}{N}}\\
 &\le \Constant \Bigl(\int_{\Rset^N} \abs{\nabla u_p}^2 + \abs{u_p}^2 \Bigr)^p,
\end{split}
\]
where the constant \(\SameConstant\) can be taken independently of \(p \in (2, \frac{N + \alpha}{N - 2})\) once \(p\) remains bounded.
It follows then that
\[
 \liminf_{p \to 2} \SameConstant \Bigl(\int_{\Rset^N} \abs{\nabla u_p}^2 + \abs{u_p}^2\Bigr)^{p - 1} \ge 1,
\]
and thus,
\[
 \liminf_{p \to 2} \int_{\Rset^N} \abs{\nabla u_p}^2 + \abs{u_p}^2 > 0.
\]

We assume now that there is a sequence \((p_n)_{n \in \N}\) such that
\begin{equation}
\label{eqPositivePartContradictionAssumption}
 \lim_{n \to \infty} \int_{\Rset^N} \abs{\nabla u_{p_n}^-}^2 + \abs{u_{p_n}^-}^2 = 0.
\end{equation}
For \(p \in (2, \frac{N + \alpha}{(N - 2)_+})\) we define the renormalised negative part
\[
 v_{p} = \frac{u_{p}^-}{\norm{u_{p}^-}_{H^1 (\Rset^N)}}.
\]
We first observe that
\begin{equation}
\label{eqRenormalizedInteractionLowerBound}
 \int_{\Rset^N} \bigl(I_{\alpha} \ast \abs{u_{p}}^{p}\bigr) \abs{v_p}^p
 = 1.
\end{equation}
By \cite{GhimentiVanSchaftingen}*{lemma~3.6},
for every \(\beta \in \bigl(\alpha, N\bigr)\) there exist \(C_5,C_6
> 0\) such that
\begin{multline*}
 \int_{\Rset^N} \bigl(I_\alpha \ast \abs{u_p}^p\bigr) \abs{v_p}^p
 \le \Constant \Bigl(\int_{\Rset^N} \abs{\nabla
u_p}^2 + \abs{u_p}^2 \int_{\Rset^N}  \abs{\nabla
v_p}^2 + \abs{v_p}^2 \Bigr)^\frac{1}{2}\\
\shoveright{\times \Bigl(\sup_{a \in \Rset^N} \int_{B_R
(a)} \abs{u_p}^\frac{2 N p}{N + \alpha}\int_{B_R (a)}
\abs{v_p}^\frac{2 N p}{N + \alpha}\Bigr)^{\frac{N + \alpha}{2
N}(1 - \frac{1}{p})}}\\
+ \frac{\Constant}{R^{\beta - \alpha}} \Bigl(\int_{\Rset^N} \abs{\nabla
u_p}^2 + \abs{u_p}^2 \int_{\Rset^N}  \abs{\nabla
v_p}^2 + \abs{v_p}^2 \Bigr)^\frac{p}{2}
\end{multline*}
The constants come from the Hardy--Littlewood--Sobolev inequality and from the Sobolev inequality;
they can thus be taken to be uniform as \(p \to 2\).
Since \(u_p\) and \(v_p\) remain bounded in \(H^1 (\Rset^N)\) as \(p \to 2\), we have
\begin{multline*}
 \int_{\Rset^N} \bigl(I_\alpha \ast \abs{u_p}^p\bigr) \abs{v_p}^p\\
 \le \Constant \Bigl(\sup_{a \in \Rset^N} \int_{B_R
(a)} \abs{u_p}^\frac{2 N p}{N + \alpha}\int_{B_R (a)}
\abs{v_p}^\frac{2 N p}{N + \alpha}\Bigr)^{\frac{N + \alpha}{2
N}(1 - \frac{1}{p})} + \frac{\Constant}{R^{\beta - \alpha}}.
\end{multline*}
In view of \eqref{eqRenormalizedInteractionLowerBound}, there exists \(R > 0\) such that
\[
 \liminf_{p \to 2} \sup_{a \in \Rset^N} \Bigl(\int_{B_R
(a)} \abs{u_p}^\frac{2 N p}{N + \alpha}\int_{B_R (a)}
\abs{v_p}^\frac{2 N p}{N + \alpha} \Bigr) > 0.
\]
In particular, there exists a sequence of vectors \((a_n)_{n \in \N}\) in \(\Rset^N\)  and a sequence of real numbers $(p_n)_{n \in \N}$ in \((2, \frac{N + \alpha}{N - 2}) \) converging to $2$
such that
\begin{equation}
\label{ineqLiminfuv}
  \liminf_{n  \to \infty} \Bigl(\int_{B_R
(a_n)} \abs{u_{p_n}}^\frac{2 N {p_n}}{N + \alpha}\int_{B_R (a_n)}
\abs{v_{p_n}}^\frac{2 N {p_n}}{N + \alpha} \Bigr) > 0.
\end{equation}
There exists thus a subsequence \((n_k)_{k \in \N}\) such that
the sequences of functions \((u_{p_{n_k}} (\cdot - a_{n_k}))_{k \in \N}\) and \((v_{p_{n_k}} (\cdot - a_{n_k}))_{k \in \N}\) both converge weakly in the space \(H^1 (\Rset^N)\) to some functions $u$ and $v\in H^1(\Rset^N)$.

By our contradiction assumption \eqref{eqPositivePartContradictionAssumption} and the Rellich--Kondrachov compactness theorem, we  have \(u \ge 0\).
By the classical Rellich--Kondrachov compactness theorem, it follows from \eqref{ineqLiminfuv} that
\begin{align}
\label{eqUPositivity}
 \int_{B_R} \abs{u}^\frac{2 N p}{N + \alpha}& > 0&
 &\text{ and }&
 \int_{B_R}
\abs{v}^\frac{2 N p}{N + \alpha} & > 0.
\end{align}
We also observe that by definition of \(v_p\),
\[
 \{x \in \Rset^N \st v_p (x) < 0\} \subseteq \{ x \in \Rset^N \st u_p (x) \le 0\},
\]
so that by the Rellich--Kondrachov theorem, we have
\begin{equation}
\label{eqUNegativity}
\{ x \in \Rset^N \st v (x) < 0\}
\subseteq \{ x \in \Rset^N \st u (x) \le 0\}.
\end{equation}

Since by the classical Rellich--Kondrachov compactness theorem, the sequence \((\abs{u_{p_{n_k}} (\cdot - a_{n_k})}^{p_n})_{k \in \N}\) converges locally in measure to \(\abs{u}^2\) and is bounded in \(L^{2 N/(N + \alpha)} (\Rset^N)\),
it converges weakly to \(\abs{u}^2\) in the space \(L^{2N/(N + \alpha)} (\Rset^N)\)
\citelist{\cite{Bogachev2007}*{proposition~4.7.12}\cite{Willem2013}*{
proposition 5.4.7}}.
In view of the Hardy--Littlewood--Sobolev inequality \eqref{eqHLS} and the continuity of bounded linear operators for the weak topology, the sequence \((I_\alpha \ast \abs{u_{p_{n_k}} (\cdot - a_{n_k})}^{p_{n_k}})_{k \in \N}\) converges weakly to \(I_\alpha \ast \abs{u}^2\)
in \(L^{2N/(N - \alpha)} (\Rset^N)\).
Since \(((\abs{u_{p_{n_k}}}^{p_{n_k} - 2} u_{p_{n_k}})(\cdot - a_{n_k}))_{k \in \N}\) converges to \(u\) in \(L^2_{\mathrm{loc}} (\Rset^N)\),
we conclude that
\[
 \bigl(I_\alpha \ast \abs{u_{p_{n_k}} (\cdot - a_{n_k})}^{p_{n_k}}\bigr)\bigl(\abs{u_{p_{n_k}}}^{p_ {n_k} - 2} u_{p_{n_k}}\bigr)\,(\cdot - a_{n_k}) \to \bigl(I_\alpha \ast \abs{u}^2\bigr)u
\]
in \(L^{2N/(2 N - \alpha)} (\Rset^N)\), as \(k \to \infty\).
By construction of the function \(u_p\), we deduce from that the function \(u \in H^1 (\Rset^N)\) is a weak solution of the problem
\[
 -\Delta u + u = \bigl(I_{\alpha} \ast \abs{u}^2\bigr) u.
\]
By the classical bootstrap method for subcritical semilinear elliptic problems applied to the Choquard equation (see for example \citelist{\cite{MorozVanSchaftingen13}*{proposition 4.1}\cite{CingolaniClappSecchi2012}*{lemma A.1}}), \(u\) is smooth. Since \(u \ge 0\), by the strong maximum principle we have either \(u=0\) or \(u > 0\), in contradiction with \eqref{eqUPositivity} and \eqref{eqUNegativity}.
The claim is thus proved by contradiction.
\end{proofclaim}

\begin{claim}
There exists \(R > 0\) such that
\[
 \limsup_{p \to 2} \sup_{a \in \Rset^N} \int_{B_R (a)} \abs{u_p^+}^\frac{2
N p}{N + \alpha} \int_{B_R (a)} \abs{u_p^-}^\frac{2 N
p}{N + \alpha} > 0.
\]
\end{claim}

\begin{proofclaim}
We assume by contradiction that for every \(R > 0\),
\begin{equation}
 \label{eqRepulsionContradiction}
 \lim_{p \to 2} \sup_{a \in \Rset^N} \int_{B_R (a)} \abs{u_p^+}^\frac{2
N p}{N + \alpha} \int_{B_R (a)} \abs{u_p^-}^\frac{2 N
p}{N + \alpha} = 0.
\end{equation}
In view of \cite{GhimentiVanSchaftingen}*{lemma~3.6} and since the sequences \((u_n^+)_{n
\in
\N}\) and \((u_n^-)_{n \in \N}\) are both bounded in \(H^1 (\Rset^N)\), we have, as in the proof of claim~\ref{claimNonzero}, for every \(\beta \in (\alpha, N)\) and \(R > 0\),
\begin{multline*}
 \int_{\Rset^N} (I_\alpha \ast \abs{u_p^+}^p) \abs{u_p^-}^p\\
 \le \Constant \Bigl(\sup_{a \in \Rset^N} \int_{B_R
(a)} \abs{u_p^+}^\frac{2 N p}{N + \alpha}\int_{B_R (a)}
\abs{u_p^-}^\frac{2 N p}{N + \alpha}\Bigr)^{\frac{N + \alpha}{2
N}(1 - \frac{1}{p})} + \frac{\Constant}{R^{\beta - \alpha}}.
\end{multline*}
By our assumption \eqref{eqRepulsionContradiction},
we have thus
\begin{equation}
\label{eqNoInteraction}
 \lim_{p \to 2} \int_{\Rset^N} \bigl(I_\alpha \ast \abs{u_p^+}^p\bigr)
\abs{u_p^-}^p = 0.
\end{equation}

We define now the pair \((t_{p, +}, t_{p, -}) \in (0, \infty)^2\) by the condition that \(t_{p, \pm} u_p^\pm \in
\mathcal{N}_{0, p}\), or equivalently,
\[
 t_{p, \pm}^{2 p - 2} =
 \frac{\displaystyle \int_{\Rset^N}\abs{\nabla u_p^\pm}^2 + \abs{u_p^\pm}^2}{
 \displaystyle  \int_{\Rset^N} \bigl(I_\alpha
\ast \abs{u_p^\pm}^p\bigr) \abs{u_p^\pm}^p}
=\frac{\displaystyle \int_{\Rset^N} \bigl(I_\alpha
\ast \abs{u_p}^p\bigr) \abs{u_p^\pm}^p}{
 \displaystyle  \int_{\Rset^N} \bigl(I_\alpha
\ast \abs{u_p^\pm}^p\bigr) \abs{u_p^\pm}^p}
=1 + o (1)
\]
as \(p \to 2\), in view of claim~\ref{claimNonzero} and of \eqref{eqNoInteraction}.
Since the family \(u_p\) remains bounded in \(H^1 (\Rset^N)\), we have
\[
 \lim_{p \to 2} \mathcal{A}_p (t_{p, +} u_p^+ + t_{p, -} u_p^-) - \mathcal{A}_p (u_p)
 = 0.
\]
In view of the identity
\begin{multline*}
 \mathcal{A}_p (t_{p, +} u_p^+ + t_{p, -} u_p^-)\\
 = \mathcal{A}_p (t_{p, +} u_p^+) + \mathcal{A}_p(t_{p, -} u_p^-)
 -\frac{t_{p, +}^p t_{p, -}^p}{p} \int_{\Rset^N} \bigl(I_\alpha \ast
\abs{u_p^+}^p\bigr) \abs{u_p^-}^p
\end{multline*}
and by \eqref{eqNoInteraction}, we conclude that
\[
  \liminf_{p \to 2} \mathcal{A}_p (u_p) \ge 2 \liminf_{p \to 2} c_{0, p}.
\]
By claim~\ref{claimConstructionSequence} and by proposition~\ref{continuityGroundstate}, this implies that
\[
 c_{\mathrm{nod}, 2} \ge 2 c_{0, 2},
\]
in contradiction with proposition~\ref{propositionStrictInequality}.
\end{proofclaim}

We are now in a position to conclude the proof.
Up to a translation, there exist $R>0$ and a sequence \((p_n)_{n \in \N}\) in \((2, \frac{N + \alpha}{N - 2})\) such that \(p_n \downto 2\)  as \(n \to \infty\),
\[
 \liminf_{n \to \infty} \int_{B_R} \abs{u_{p_n}^\pm}^\frac{2
N p}{N + \alpha} \ge 0
\]
and the sequence \((u_{p_n})_{n \in \N}\) converges weakly in \(H^1 (\Rset^N)\) to some function \(u \in H^1
(\Rset^N)\).

As in the proof of claim~\ref{claimNonzero}, by the
weak convergence and by the classical Rellich--Kondrachov compactness
theorem, we have \(\mathcal{A}'_2 (u) = 0\) and \(u^\pm \ne 0\), whence \(u \in
\mathcal{N}_{2, \mathrm{nod}}\).
We also have by the weak lower semicontinuity of the norm,
\[
\begin{split}
\lim_{n \to \infty} \mathcal{A}_{p_n} (u_{p_n}) &= \lim_{n \to \infty} \Bigl(\frac{1}{2} - \frac{1}{2 p_n} \Bigr)\int_{\Rset^N}
\abs{\nabla u_{p_n}}^2 + \abs{u_{p_n}}^2\\
&\ge \frac{1}{4} \int_{\Rset^N} \abs{\nabla u}^2 +
\abs{u}^2 =\mathcal{A}_2 (u),
\end{split}
\]
and thus \(\mathcal{A}_2 (u) = c_{\mathrm{nod}, 2}\).
\end{proof}

In claim~\ref{claimNonzero}, the study of the renormalised negative part to prevent vanishing is reminiscent of the idea of taking the renormalised approximate solution to bypass the Ambrosetti--Rabinowitz superlinearity condition \citelist{\cite{LiuWang2004}\cite{VanSchaftigenWillem2008}}.


\begin{bibdiv}
  \begin{biblist}

  \bib{AckermannWeth2005}{article}{
   author={Ackermann, Nils},
   author={Weth, Tobias},
   title={Multibump solutions of nonlinear periodic Schr\"odinger equations
   in a degenerate setting},
   journal={Commun. Contemp. Math.},
   volume={7},
   date={2005},
   number={3},
   pages={269--298},
   issn={0219-1997},
}


    \bib{Bogachev2007}{book}{
      author={Bogachev, V. I.},
      title={Measure theory},
      publisher={Springer},
      place={Berlin},
      date={2007},
      isbn={978-3-540-34513-8},
      isbn={3-540-34513-2},
    }

\bib{Ghimenti-soliton}{article}{
   author={Bonanno, Claudio},
   author={d'Avenia, Pietro},
   author={Ghimenti, Marco},
   author={Squassina, Marco},
   title={Soliton dynamics for the generalized Choquard equation},
   journal={J. Math. Anal. Appl.},
   volume={417},
   date={2014},
   number={1},
   pages={180--199},
   issn={0022-247X},
}

    \bib{CastroCossioNeuberger1997}{article}{
      author={Castro, Alfonso},
      author={Cossio, Jorge},
      author={Neuberger, John M.},
      title={A sign-changing solution for a superlinear Dirichlet problem},
      journal={Rocky Mountain J. Math.},
      volume={27},
      date={1997},
      number={4},
      pages={1041--1053},
      issn={0035-7596},
    }

    \bib{CastorCossioNeuberger1998}{article}{
      author={Castro, Alfonso},
      author={Cossio, Jorge},
      author={Neuberger, John M.},
      title={A minmax principle, index of the critical point, and existence of
      sign-changing solutions to elliptic boundary value problems},
      journal={Electron. J. Differential Equations},
      volume={1998},
      date={1998},
      number={2},
      pages={18},
      issn={1072-6691},
    }


    \bib{CeramiSoliminiStruwe1986}{article}{
      author={Cerami, G.},
      author={Solimini, S.},
      author={Struwe, M.},
      title={Some existence results for superlinear elliptic boundary value
      problems involving critical exponents},
      journal={J. Funct. Anal.},
      volume={69},
      date={1986},
      number={3},
      pages={289--306},
      issn={0022-1236},
    }

    \bib{CingolaniClappSecchi2012}{article}{
      author={Cingolani, Silvia},
      author={Clapp, M{\'o}nica},
      author={Secchi, Simone},
      title={Multiple solutions to a magnetic nonlinear Choquard equation},
      journal={Z. Angew. Math. Phys.},
      volume={63},
      date={2012},
      number={2},
      pages={233--248},
      issn={0044-2275},
    }

    \bib{CingolaniSecchi}{article}{
   author={Cingolani, Silvia},
   author={Secchi, Simone},
   title={Multiple \(\mathbb{S}^{1}\)-orbits for the Schr\"odinger-Newton system},
   journal={Differential and Integral Equations},
   volume={26},
   number={9/10},
   pages={867--884},
   date={2013},
}

\bib{CingolaniSecchi-2015}{article}{
   author={Cingolani, Silvia},
   author={Secchi, Simone},
   title={Ground states for the pseudo-relativistic Hartree equation with
   external potential},
   journal={Proc. Roy. Soc. Edinburgh Sect. A},
   volume={145},
   date={2015},
   number={1},
   pages={73--90},
   issn={0308-2105},
}

    \bib{ClappSalazar}{article}{
   author={Clapp, M{\'o}nica},
   author={Salazar, Dora},
   title={Positive and sign changing solutions to a nonlinear Choquard
   equation},
   journal={J. Math. Anal. Appl.},
   volume={407},
   date={2013},
   number={1},
   pages={1--15},
   issn={0022-247X},
}

\bib{fractionalChoquard}{article}{
   author={d'Avenia, Pietro},
   author={Siciliano, Gaetano},
   author={Squassina, Marco},
   title={On fractional Choquard equations},
   journal={Math. Models Methods Appl. Sci.},
   volume={25},
   date={2015},
   number={8},
   pages={1447--1476},
   issn={0218-2025},
}

\bib{DevreeseAlexandrov2009}{book}{
  author={Devreese, Jozef T.},
  author={Alexandrov, Alexandre S. },
  title={Advances in polaron physics},
  year={2010},
  publisher={Springer},
  volume={159},
  series={Springer Series in Solid-State Sciences},
  pages={ix+167},
}


\bib{Frohlich-Lenzmann}{article}{
   author={Fr{\"o}hlich, J{\"u}rg},
   author={Lenzmann, Enno},
   title={Mean-field limit of quantum Bose gases and nonlinear Hartree
   equation},
   conference={
      title={S\'eminaire: \'Equations aux D\'eriv\'ees Partielles.
      2003--2004},
   },
   book={
      series={S\'emin. \'Equ. D\'eriv. Partielles},
      publisher={\'Ecole Polytech., Palaiseau},
   },
   date={2004},
   pages={Exp. No. XIX, 26},
}

    \bib{GhimentiVanSchaftingen}{article}{
      author={Ghimenti, Marco},
      author={Van Schaftingen, Jean},
      title={Nodal solutions for the Choquard equation},
      eprint={arXiv:1503.06031},
    }


\bib{KRWJones1995newtonian}{article}{
  title={Newtonian Quantum Gravity},
  author={Jones, K. R. W.},
  journal={Australian Journal of Physics},
  volume={48},
  number={6},
  pages={1055-1081},
  year={1995},
}

\bib{Lenzmann-2009}{article}{
   author={Lenzmann, Enno},
   title={Uniqueness of ground states for pseudorelativistic Hartree
   equations},
   journal={Anal. PDE},
   volume={2},
   date={2009},
   number={1},
   pages={1--27},
   issn={1948-206X},
}

    \bib{Lieb1977}{article}{
      author={Lieb, Elliott H.},
      title={Existence and uniqueness of the minimizing solution of Choquard's
      nonlinear equation},
      journal={Studies in Appl. Math.},
      volume={57},
      date={1976/77},
      number={2},
      pages={93--105},
    }
    \bib{LiebLoss2001}{book}{
      author={Lieb, Elliott H.},
      author={Loss, Michael},
      title={Analysis},
      series={Graduate Studies in Mathematics},
      volume={14},
      edition={2},
      publisher={American Mathematical Society},
      place={Providence, RI},
      date={2001},
      pages={xxii+346},
      isbn={0-8218-2783-9},
    }

    \bib{Lions1980}{article}{
      author={Lions, P.-L.},
      title={The Choquard equation and related questions},
      journal={Nonlinear Anal.},
      volume={4},
      date={1980},
      number={6},
      pages={1063--1072},
      issn={0362-546X},
    }

    \bib{Lions1984-1}{article}{
   author={Lions, P.-L.},
   title={The concentration-compactness principle in the calculus of
   variations. The locally compact case.},
   part = {I},
   journal={Ann. Inst. H. Poincar\'e Anal. Non Lin\'eaire},
   volume={1},
   date={1984},
   number={2},
   pages={109--145},
   issn={0294-1449},
}

     \bib{Lions1984CC2}{article}{
      author={Lions, Pierre-Louis},
      title={The concentration-compactness principle in the calculus of
      variations. The locally compact case. II},
      journal={Ann. Inst. H. Poincar\'e Anal. Non Lin\'eaire},
      volume={1},
      date={1984},
      number={4},
      pages={223--283},
      issn={0294-1449},
    }

    \bib{LiuWang2004}{article}{
      author={Liu, Zhaoli},
      author={Wang, Zhi-Qiang},
      title={On the Ambrosetti-Rabinowitz superlinear condition},
      journal={Adv. Nonlinear Stud.},
      volume={4},
      date={2004},
      number={4},
      pages={563--574},
      issn={1536-1365},
    }

\bib{Ma-Zhao-2010}{article}{
   author={Ma,Li},
   author={Zhao, Lin},
   title={Classification of positive solitary solutions of the nonlinear
   Choquard equation},
   journal={Arch. Ration. Mech. Anal.},
   volume={195},
   date={2010},
   number={2},
   pages={455--467},
   issn={0003-9527},
}

\bib{Menzala1980}{article}{
   author={Menzala, Gustavo Perla},
   title={On regular solutions of a nonlinear equation of Choquard's type},
   journal={Proc. Roy. Soc. Edinburgh Sect. A},
   volume={86},
   date={1980},
   number={3-4},
   pages={291--301},
   issn={0308-2105},
}

\bib{Moroz-Penrose-Tod-1998}{article}{
   author={Moroz, Irene M.},
   author={Penrose, Roger},
   author={Tod, Paul},
   title={Spherically-symmetric solutions of the Schr\"odinger-Newton
   equations},
   journal={Classical Quantum Gravity},
   volume={15},
   date={1998},
   number={9},
   pages={2733--2742},
   issn={0264-9381},
}



    \bib{MorozVanSchaftingen13}{article}{
      author={Moroz, Vitaly},
      author={Van Schaftingen, Jean},
      title={Groundstates of nonlinear Choquard equations: existence, qualitative
        properties and decay asymptotics},
      journal={J. Funct. Anal.},
      volume={265},
      date={2013},
      pages={153--184},
    }


\bib{Pekar1954}{book}{
   author={Pekar, S.},
   title={Untersuchung {\"u}ber die Elektronentheorie der Kristalle},
   publisher={Akademie Verlag},
   place={Berlin},
   date={1954},
   pages={184},
}


    \bib{VanSchaftingen2014}{article}{
      author={Van Schaftingen, Jean},
      title={Interpolation inequalities between Sobolev and Morrey--Campanato
      spaces: A common gateway to concentration-compactness and
      Gagliardo--Nirenberg interpolation inequalities},
      journal={Port. Math.},
      volume={71},
      date={2014},
      number={3-4},
      pages={159--175},
      issn={0032-5155},
    }

    \bib{VanSchaftigenWillem2008}{article}{
      author={Van Schaftingen, Jean},
      author={Willem, Michel},
      title={Symmetry of solutions of semilinear elliptic problems},
      journal={J. Eur. Math. Soc. (JEMS)},
      volume={10},
      date={2008},
      number={2},
      pages={439--456},
      issn={1435-9855},
    }

    \bib{Weth2001}{thesis}{
      author={Weth, Tobias},
      title={Spectral and variational characterizations
	of solutions to semilinear eigenvalue problems},
      date={2001},
      organization={Johannes Gutenberg-Universit\"at},
      address={Mainz},
    }
    \bib{Weth2006}{article}{
   author={Weth, Tobias},
   title={Energy bounds for entire nodal solutions of autonomous superlinear
   equations},
   journal={Calc. Var. Partial Differential Equations},
   volume={27},
   date={2006},
   number={4},
   pages={421--437},
   issn={0944-2669},
}

    \bib{Willem1996}{book}{
      author={Willem, Michel},
      title={Minimax theorems},
      series={Progress in Nonlinear Differential Equations and their
      Applications, 24},
      publisher={Birkh\"auser},
      address={Boston, Mass.},
      date={1996},
      pages={x+162},
      isbn={0-8176-3913-6},
    }

    \bib{Willem2013}{book}{
      author = {Willem, Michel},
      title = {Functional analysis},
      subtitle = {Fundamentals and Applications},
      series={Cornerstones},
      publisher = {Birkh\"auser},
      place = {Basel},
      volume = {XIV},
      pages = {213},
      date={2013},
    }







 \end{biblist}

\end{bibdiv}
\end{document}